\documentclass{amsart}
\usepackage[title]{appendix}
\usepackage[a4paper,top=3cm,bottom=2cm,left=3cm,right=3cm,marginparwidth=1.75cm]{geometry}
\usepackage[noabbrev,capitalise]{cleveref}
\usepackage{amsmath,amssymb,amsthm,amsfonts,mathrsfs,mathtools,relsize}
\usepackage{xcolor}
\usepackage{thm-restate}
\title{Dense minors of graphs with independence number two}

\author{Sergey Norin}
\address{Department of Mathematics and Statistics, McGill University, Montreal, QC, Canada}
\author{Paul Seymour}\address{Princeton University, Princeton, NJ 08544, United States}

\thanks{Norin is supported by an NSERC Discovery Grant. Seymour is supported by AFOSR grant A9550-19-1-0187.}

\newtheorem{thm}{Theorem}[section]

\newtheorem{lem}[thm]{Lemma}

\newtheorem{cor}[thm]{Corollary}

\newtheorem*{theorem*}{Theorem}

\theoremstyle{remark}

\newtheoremstyle{named}{}{}{\itshape}{}{\bfseries}{.}{.5em}{\thmnote{#3 }#1}
\theoremstyle{named}

\usepackage[shortlabels]{enumitem}

\newcommand{\mc}[1]{\mathcal{#1}}

\newcommand{\bb}[1]{\mathbb{#1}}

\newcommand{\brm}[1]{\operatorname{#1}}

\newcommand{\fs}[2]{\left(\frac{#1}{#2}\right)}
\newcommand{\s}[1]{\left(#1\right)}

\newcommand{\Var}{\brm{Var}}
\begin{document}

\begin{abstract} Motivated by Hadwiger's conjecture, we prove that every $n$-vertex graph $G$ with no independent set of size three  contains an  $\lceil n/2\rceil$-vertex simple minor $H$  with 
	$$0.98688 \cdot \binom{|V(H)|}{2} - o(n^2)$$
	edges.
\end{abstract}
\maketitle

\section{Introduction}

All the graphs in this paper are simple.
In 1943 Hadwiger conjectured \cite{Had43} that every graph with chromatic number at least $t$ contains the complete graph $K_t$ on $t$ vertices as a minor. Hadwiger's conjecture remains wide open for every $t \geq 7$ and many of its relaxations  have been investigated. (See \cite{Sey16Survey} for a survey.) In particular, over the years increasingly strong  lower bounds have been obtained on the function $f(t)$ such that every  graph with chromatic number at least $t$ is guaranteed to contain $K_{f(t)}$ as a minor. The current  record $f(t)=\Omega(\frac{t}{\log\log t})$ is due to Delcourt and Postle~\cite{DelPos21}.

Let $\alpha(G)$ denote the maximum size of an independent set in $G$. 
Hadwiger's conjecture implies that every graph $G$ with $\alpha(G) \leq 2$ contains $K_{\lceil |V(G)|/2 \rceil}$ as a minor. It remains open and still very challenging in this case. A classical result of Duchet and Meyniel~\cite{DucMey82} implies that  every graph $G$ with $\alpha(G)=2$ contains $K_{\lceil |V(G)|/3 \rceil}$ as a minor.  A well-known open problem (see e.g. \cite[Question 4.8]{Sey16Survey}) asks for a constant factor improvement of their bound. 

\vskip 5pt

We consider relaxing Hadwiger's conjecture in a different direction, investigating maximum $g(t)$ such that every graph with chromatic number at least $t$ contains a minor with $t$ vertices and $g(t)$ edges. (Hadwiger's conjecture asserts that $g(t)=\binom{t}{2}$.) 

While it is not known whether $f(t)$ is linear, the function $g(t)$ is quadratic. Indeed, every graph with chromatic number at least $t$ contains a subgraph with minimum degree at least $t-1$. Mader~\cite{Mader68} showed that every graph with 
average degree at least $t-1$ contains a minor with at most $t$ vertices and minimum degree at least $t/2$,
and hence with at least $t(t+2)/8$ edges; and
consequently there is a minor with exactly $t$ vertices and at least $t(t+2)/8$
edges. It follows that $g(t) \geq \frac{t(t+2)}{8}=\frac{1}{4}\binom{t}{2} + o(t^2)$. The factor $1/4$ can likely be significantly improved. However it appears to be quite challenging, for example, to show that every graph  with chromatic number at least $t$ contains a minor on $t$ vertices with at most one tenth of all possible edges missing.

We concentrate on graphs $G$ with  $\alpha(G) \leq 2$. Our main result, stated below, guarantees in such graphs the 
existence of a minor with  $\lceil |V(G)|/2 \rceil$ vertices and with fewer than $1/76$ of all possible edges missing.

\begin{thm}\label{t:main}
	Let $G$ be a graph with $\alpha(G) \leq 2$.  Then there exists a  minor $H$ of $G$ with $|V(H)|= \lceil |V(G)|/2 \rceil$ such that
	\begin{equation}\label{e:main}|{E}(H)| \geq  \s{\gamma - o(1)}\binom{|V(H)|}{2},
	\end{equation}
	where 
	$$ \gamma = 1-\max_{ z \in [0,1/4]}\frac{z^3 (5 - 38 z + 92 z^2 - 80 z^3)}{(1 - 3 z + 3 z^2)^2}=0.986882\ldots$$
\end{thm}

\section{Proof of \cref{t:main}}

We start with a brief outline of the proof. 
Let $G$ be a graph with $\alpha(G) \leq 2$. We may assume that $|V(G)|$ is even. 
Let $\omega(G)$ denote the maximum size of cliques in $G$. If $\omega(G) \geq |V(G)|/4$ then $G$ contains a $K_{|V(G)|/2}$ minor by a result of   Chudnovsky and Seymour~\cite{ChuSey12}, which we state below. Thus we assume  $\omega(G) < |V(G)|/4$. 

A \emph{seagull} in $G$ is an induced three-vertex path. Note that  every 
vertex of $G$ has a neighbour in every seagull, and so by contracting the edges of a seagull, we obtain a vertex complete to the rest of the graph. Also note that the set of non-neighbours of every vertex forms a clique, and so every vertex of $G$ has at most $\omega(G)$ non-neighbours. We repeatedly use these observations.

We find a minor $H$ satisfying the theorem by choosing an arbitrary maximum clique $Z$,  partitioning the remaining vertices into a matching of size $|V(G)|/2-2\omega(G)$ and a collection of seagulls, and contracting the edges of the matching and of the seagulls. 
We find the matching first,  
randomly, ensuring that there are not too many non-adjacencies between pairs of matching edges and between matching edges and $Z$.
There are exactly $3\omega(G)$ vertices of $V(G)-Z$ that are left uncovered by the matching, and we partition them into seagulls.

\vskip 10pt
We now move on to details.
The last part of the above argument uses a characterization of graphs $G$  with $\alpha(G) \leq 2$ containing $k$ pairwise disjoint seagulls due to Chudnovsky and Seymour~\cite{ChuSey12}. To state this charactezization we need a couple of definitions.

Let $C$ be a clique in the graph $G$. Let $X$ consist of all vertices in $V(G)-C$ with a neighbour and a non-neighbour in $C$. The \emph{capacity} $\brm{cap}(C)$ of $C$ is defined to be equal to $\frac{|V(G)-C|+|X|}{2}$.
A \emph{five-wheel} is a six-vertex graph obtained from a cycle of
length five by adding one new vertex adjacent to every vertex of the cycle. We denote by $\bar{G}$ the complement of a graph $G$.

\begin{thm}[{\cite[1.6]{ChuSey12}}]\label{t:seagull}
	Let $G$ be a graph with $\alpha(G) \leq 2$, let $k \geq 0$ be an integer.  Then   $G$ contains $k$ pairwise disjoint seagulls if and only if
	\begin{description}
		\item[(i)] $|V(G)| \geq 3k$,
		\item[(ii)] $G$ is $k$-connected,
		\item[(iii)] $\brm{cap}(C) \geq k$ for every clique $C$ in $G$,
		\item[(iv)] $\bar{G}$ contains a matching of size $k$,
		\item[(v)] if $k=2$ then $G$ is not a five-wheel. 
	\end{description}
\end{thm}

We will use an easy-to-apply corollary of \cref{t:seagull}.

\begin{cor}\label{l:seagull}
	Let $k \geq 0$ be an integer. Let $G$ be a graph with $\alpha(G) \leq 2$,   $|V(G)|=3k$ and $\omega(G) \leq k$;
 then $G$ has $k$ pairwise disjoint seagulls.  
\end{cor}

\begin{proof}
	We will verify that conditions (i)-(v) of \cref{t:seagull} hold. Condition (i) is immediate. 
	
	Suppose for a contradiction that (ii) does not hold. Then there exists $X \subseteq V(G)$ such that 
	$|X| < k$ and $G \setminus X$ is not connected. It follows that $G \setminus X$ is a union of two cliques. As $|V(G \setminus X)| \geq 2k+1$,   one of these cliques must have size at least $k+1$. This gives the desired contradiction, as $\omega(G) \leq k$.
	
	For every clique $C$ in $G$ we have  $|C| \leq k$,  and so $\brm{cap}(C) \geq \frac{|V(G)| - |C|}{2} \geq k$. Thus (iii) holds.
	
	Let $M$ be a maximal matching in $\bar{G}$. Then the vertices of $G$ incident with no edge of $M$ form a clique, implying that $|V(G)|-2|M| \leq k$. Thus $|M| \geq k$ and (iv) holds.
	
	Finally, (v) holds, as a five-wheel contains a clique of size three.  
\end{proof}

Additionally,  as mentioned above, we use another consequence of  \cref{t:seagull} established in~\cite{ChuSey12}.

\begin{thm}[{\cite[1.3]{ChuSey12}}]\label{t:seagull2}
	Let $k \geq 0$ be an integer. Let $G$ be a graph with $\alpha(G) \leq 2$, with an even number of vertices and 
with $\omega(G) \geq |V(G)|/4$. Then $G$ contains a $K_{|V(G)|/2}$ minor.
\end{thm}

Next we make a couple of observations about the procedure that we use to generate a random matching. 

Let $X$ be a finite set with $|X|$ even. We denote by $X^{(2)}$ the set of all two-element subsets of $X$. Let $\mc{M}_X$, or simply $\mc{M}$ for brevity, denote a partition of $X$ into pairs, chosen uniformly at random. 
Note that \begin{equation}\label{e:MBounds} \Pr[e \in \mc{M}_X] =\frac{1}{|X|-1} \qquad \text{and} \qquad \Pr[e,f \in \mc{M}_X] = \frac{1}{(|X|-1)(|X|-3)}\end{equation} for every $e \in X^{(2)}$ and every $f \in X^{(2)}$ disjoint from $e$. 
 
We need a bound on the concentration of $|F \cap \mc{M}|$ for   $F \subseteq X^{(2)}$.
\begin{lem}\label{l:Chebyshev}
	Let $X$ be a finite set with $|X| \geq 4$ even. Then for every $F \subseteq X^{(2)}$ and every $\lambda>0$ we have
	\begin{equation}
	\Pr\left[\left||F \cap \mc{M}_X| - \frac{|F|}{|X|-1}\right| \geq \lambda \right] \leq \frac{|X|}{\lambda^2}.
	\end{equation}
\end{lem}

\begin{proof} We may assume that  
\begin{equation}\label{e:FUpper} |F| \leq \frac{|X|(|X|-1)}{4}\end{equation} 
by replacing $F$ with $X^{(2)} -F$ if needed. For $e \in F$, let $Z_e$ be the indicator random variable with $Z_e=1$ if $e \in \mc{M}_X$ and $Z_e=0$, otherwise.  Let $Z =\sum_{e \in F}Z_e = |F \cap \mc{M}_X|$. By \eqref{e:MBounds} we have $E[Z] = |F|/(|X|-1)$ and
	\begin{align*}
	\Var[Z] &= \sum_{(e,f) \in F \times F} \brm{Cov}[Z_e,Z_f] \\  
&\leq \sum_{e \in F} \brm{Cov}[Z_e,Z_e] + \sum_{e \in F  }\sum_{f \in F, f \cap e = \emptyset} \brm{Cov}[Z_e,Z_e] \\ 
&\stackrel{\eqref{e:MBounds}}{\leq} 
\frac{|F|}{|X|-1} + |F|^2\s{ \frac{1}{(|X|-1)(|X|-3)} - \s{\frac{1}{|X|-1}}^2} \\ 
&\stackrel{\eqref{e:FUpper}}{\leq} \frac{|X|}{4} + \frac{|X|^2}{8(|X|-3)} \\ 
&\leq |X|.
	\end{align*} 	
	
	 Thus by Chebyshev's inequality (see for instance~\cite[Theorem 4.1.1]{AloSpe16})  $$	\Pr\left[\left| |F \cap \mc{M}_X| - \frac{|F|}{|X|-1}\right| \geq \lambda \right] \leq \Pr[||Z| - E[Z]| \geq \lambda] \leq \frac {\Var[Z]}{\lambda^2} \leq \frac{|X|}{\lambda^2},$$
as desired.
\end{proof}	

We are now ready for the main technical lemma.

\begin{lem}\label{l:main}
	Let $G$ be a graph  with $|V(G)| \geq 6$ and even, with $\alpha(G) \leq 2$, and with $\omega(G) < |V(G)|/4$; and define
$n=|V(G)|/2$ and $k=\omega(G)$. 
Let $Z$ be a clique of $G$ with $|Z|=k$, and define $a=\sum_{v \in Z} \deg_{\bar{G}}(v) $ and $b = |E(\bar{G} \setminus Z)|$. 
	 Let $0 < \lambda \leq \frac{k-1}{2}$ with $\lambda^2>2n$, and let 
$$p = \frac{n-2k}{n-\frac{k+1}{2} - \frac{b}{2n-k-2}-\lambda}.$$ 
Then there is a minor $H$ of $G$ with $|V(H)|= n$ and
	\begin{equation}\label{e:J}\binom{n}{2}-|E(H)| \leq \frac{1}{1-\frac{2n}{\lambda^2}} \s{\frac{b(k-1)^2p^2}{4(2n-k-2)(2n-k-4)}+ \frac{a(k-1)p}{2(2n-k-2)}}.
	\end{equation}
\end{lem}
\begin{proof}
	Let the graph $G'$ be obtained from $G$ by deleting $Z$ and at most one other vertex, so that $|V(G')| \geq 2n-k-1$ is even. Let $x=|V(G')|$. Note that $x \geq \lceil  3n/2 \rceil -1 \geq 4$. 
	
	 Let $\mc{A}$ be the set of all partitions $M$ of $V(G')$ into pairs that satisfy    
\begin{equation}\label{e:a}|M \cap E(G')| \geq \frac{|E(G')|}{x-1} - \lambda. \end{equation} 
Note that $\deg_{\bar{G}}(v) \leq k$ for every $v \in V(G)$. Thus $\deg_{G'}(v) \geq x - k-1$ for every $v \in V(G')$, and so $|E(G')| \geq \frac{x(x - k-1)}{2}$.
	As $\lambda \leq \frac{k-1}{2}$, it follows that for every $M \in \mc{A}$ we have
\begin{align}\label{e:MH}  |M \cap E(G')| &\geq \frac{x(x - k-1)}{2(x-1)}-\lambda \geq  \frac{x}{2} - \frac{kx}{2(x-1)} - \frac{k-1}{2} \geq n-2k\end{align} 
(since $x\ge 2n-k-1$ and $x/(x-1)\le 4/3$).
Moreover, since $|E(G')|+b\ge x(x-1)/2$, it follows that for every $M \in \mc{A}$
	 \begin{align}\label{e:MH2} |M \cap E(G')| &\geq \frac{x}{2} - \frac{b}{x-1}-\lambda \geq n-\frac{k+1}{2} - \frac{b}{2n-k-2}-\lambda. \end{align}

	 Let $\mc{M}=\mc{M}_{V(G')}$ be a partition of $V(G')$ into pairs chosen uniformly at random. 
Let $q = 1-\frac{2n}{\lambda^2}$. Thus $q>0$. 
By \cref{l:Chebyshev}, applied with $X=V(G')$ and $F=E(G')$, we have \begin{equation}\label{e:q}\Pr[\mc{M} \in \mc{A}] \geq 1-\frac{x}{\lambda^2} \geq q.\end{equation}  
	
	Let $\mc{M}^{*}$ be a random matching in $G'$, obtained by choosing $M \in \mc{A}$ uniformly at random  and then  
choosing a subset of $n-2k$ edges of $M \cap E(G')$, also uniformly at random. Such a selection is possible by \eqref{e:MH}; and 
once $M$ is chosen, every edge of $M \cap E(G')$ is selected with probability at most $\frac{n-2k}{|M \cap E(G')|} \leq p$ 
by \eqref{e:MH2}. We have
$$\Pr[e\in \mc{M}]\ge \Pr[e\in \mc{M}\text{ and } \mc{M}\in \mc{A}]=\Pr[e\in M]\cdot\Pr[\mc{M} \in \mc{A}];$$
so using \eqref{e:MBounds} and \eqref{e:q}, we obtain
	\begin{equation}\label{e:P}\Pr[e \in \mc{M}^{*}] \le p\cdot \Pr[e\in M] \leq p \cdot \frac{\Pr[e \in \mc{M}]}{\Pr[\mc{M} \in \mc{A}]} \leq \frac{p}{q(x-1)}\end{equation} for every $e \in E(G')$.
	Similarly, once $M \in \mc{A}$ is chosen, a given pair of edges in $M \cap E(G')$ is selected to be in $\mc{M}^{*}$ with probability at most $p^2$, and thus \begin{equation}\label{e:P2}\Pr[e,f \in \mc{M}^{*}]   \leq \frac{p^2}{q(x-1)(x-3)}\end{equation} 
	for every pair of distinct $e,f \in E(G')$.

	Next we construct a minor $H$ of $G$ using $\mc{M}^{*}$. Let $S$ be the set of all vertices in $V(G)-Z$ incident with 
no edges of  $\mc{M}^{*}$. Then $|S|=3k$, and so there exists a collection $\mc{S}$ of $k$ pairwise disjoint seagulls in $G$ 
with all vertices in $S$, by \cref{l:seagull} applied to the subgraph of $G$ induced by $S$. We now obtain $H$ by contracting all the edges of $\mc{M}^*$ and of the seagulls in $\mc{S}$. The rest of the proof consists of showing that the expectation of
	$\binom{n}{2}-|E(H)| = |E(\bar{H})|$ satisfies \eqref{e:J}. 
	
	We say that a subset $Q \subseteq V(G')$ is \emph{a bad quadruple} if $|Q|=4$ and the subgraph of $G'$ induced on $Q$ is a 
matching of size two.  A subset $T \subseteq V(G)$ is \emph{a bad triple} if $|T|=3$, there exists a unique vertex $z \in T \cap Z$,
 and $z$ is not adjacent to the other vertices of $T$.
	
	Every vertex of $G$ has a neighbour in every seagull, and the vertices of $Z$ are pairwise adjacent.
	Thus every pair of non-adjacent vertices of $H$ corresponds to either a bad triple which includes an edge of $\mc{M}^{*}$ or a bad quadruple which consists of the union of vertex sets of two edges of $\mc{M}^{*}$. We finish the proof by 
bounding above the expected number of bad triples and quadruples of this form.
	
	First, we give an upper bound for the total number of bad quadruples by considering ordered sequences $(u,v,w,z)$ of distinct vertices of $G'$ such that 
$\{u,v,w,z\}$ is a  bad quadruple and $uw,vz \in E(G')$. Note that every bad quadruple corresponds to eight such 
ordered sequences. There are at most $2|E(\bar{G'})|\leq 2b$ ways to choose the non-adjacent pair $(u,v)$, and for fixed 
$(u,v)$, the vertex $w \in V(G')-\{u\}$ must be one of at most $k-1$ remaining non-neighbours of $v$, and $z$ must be chosen among at 
most $k-1$ non-neighbours of $v$.   Thus there at most $2b(k-1)^2$ ordered sequences  $(u,v,w,z)$ as above, and so at most 
$b(k-1)^2/4$ bad quadruples. By \eqref{e:P2} a given bad quadruple contains two edges of $\mc{M}^*$ with probability at most $\frac{p^2}{q(x-1)(x-3)}$. It follows that  the expected number of pairs of non-adjacent vertices of $H$ coming from bad quadruples is at most $$\frac{1}{q}\cdot \frac{b(k-1)^2p^2}{4(x-1)(x-3)}.$$
   
    Similarly, we give an upper bound for the total number of bad triples by considering ordered triples $(z,v,w)$ such that 
$z \in Z$ and $v,w \in V(G')$ are distinct non-neighbours of $z$.  There are $a$ ways of choosing the pair $(z,v)$,  and at most $k-1$ ways of choosing the second non-neighbour $w$ of $z$. As every bad triple corresponds to two 
such ordered triples, there are at most $a(k-1)/2$ bad triples, and using \eqref{e:P} it follows that the expected number of pairs 
of non-adjacent vertices of $H$ coming from bad triples is at most 
    $$\frac{1}{q} \cdot \frac{a(k-1)p}{2(x-1)}.$$
    As $x \geq 2n-k-1$, it follows that 
   \begin{align*}\bb{E}[|E(\bar{H})|] &\leq \frac{1}{q}\s{\frac{b(k-1)^2p^2}{4(x-1)(x-3)} + \frac{a(k-1)p}{2(x-1)}} \\ 
&\leq \frac{1}{1-\frac{2n}{\lambda^2}} \s{\frac{b(k-1)^2p^2}{4(2n-k-2)(2n-k-4)}+ \frac{a(k-1)p}{2(2n-k-2)}}, \end{align*}
	as desired.	
\end{proof}

It remains to optimize the bound in \cref{l:main} for large $n$.

\begin{proof}[Proof of \cref{t:main}]
We may assume that the number of vertices of $G$ is even, without loss of generality. To see this, suppose that the theorem 
holds for graphs with an even number of vertices, while  the number of vertices of $G$ is odd. Choose $v\in V(G)$; there 
is a minor $H$ of  $G \setminus v$ on $\lceil |V(G)|/2\rceil-1$ vertices $|{E}(H')| \geq  \s{\gamma - o(1)}\binom{|V(H')|}{2}.$ 
By adding the vertex $v$ to $H'$ we obtain a minor $H$ of $G$ satisfying the theorem.

Let $n = |V(G)|/2$ and let $k=\omega(G)$. If $k \geq n/2$ then $G$ contains a $K_n$ minor by \cref{t:seagull2}.  
Thus we may assume that $k < n/2$ and apply \cref{l:main} with $\lambda   = n^{2/3}$.  Let $Z,a,b,p$ and $H$ be as in \cref{l:main}. We will show that $H$ satisfies the theorem. 
 
Certainly $\frac{2n}{\lambda^2} = o(1)$. We have $a+2b = \sum_{v \in V(G)-Z} \deg_{\bar{G}}(v) \leq k(2n-k)$, and so
$ b \leq \frac{k(2n-k)-a}{2}$, 
and 
$$p \leq \frac{n-2k}{n-k+\frac{a}{2(2n-k)}} + o(1).$$ 
Let $z=\frac{k}{2n}$ and let $\zeta = \frac{a}{4n^2}$. It follows that
$ b \leq 2n^2((1-z)z-\zeta)$, and 
$$p\le \frac{1-4z}{1-2z+\frac{\zeta}{1-z}}+o(1).$$
Plugging the bounds above into \eqref{e:J}  yields 

\begin{align}\binom{n}{2}-|E(H)| 
&\leq \s{\frac{((1-z)z-\zeta)z^2}{2(1-z)^2}\fs{1-4z}{1-2z+\frac{\zeta}{1-z}}^2+ \frac{2\zeta z}{(1-z)}\fs{1-4z}{1-2z+\frac{\zeta}{1-z}}}n^2 + o(n^2) \notag\\
&=
 \frac{z (1- 4 z) ( z^2 (1 - 5 z + 4 z^2) +
 \zeta (4 - 13 z + 12 z^2)+ 4 \zeta^2)}{ (1 +  \zeta - 3 z + 2 z^2)^2}\binom{n}{2}+ o(n^2) \label{e:J2}
 \end{align}


We have $z \leq 1/4$, as $k\leq n/2$. Moreover, $ a = \sum_{v \in Z} \deg_{\bar{G}}(v) \leq k^2$,  and so $\zeta \leq z^2$.  It is straightforward to verify that the last expression in \eqref{e:J2} increases with $\zeta$ for $0 \leq \zeta \leq z^2 \leq 1/16$. Substituting $\zeta=z^2$ yields 
$$|E(H)| \geq \s{1- \frac{z^3 (5 - 38 z + 92 z^2 - 80 z^3)}{(1 - 3 z + 3 z^2)^2}}\binom{n}{2}- o(n^2). 
$$
The coefficient of $\binom{n}{2}$ on the right side is minimized (for $0\le z\le 1/4$) when $z\sim .193984$, and its value there
is approximately
$0.986882$, which implies \eqref{e:main}.	
\end{proof}

\subsubsection*{Acknowledgements.} This research was partially completed at a workshop held at the Bellairs Research Institute in Barbados in March 2022. We thank the participants of the workshop for creating a stimulating working environment, and Bojan Mohar, Yingjie Qian, St\'{e}phan Thomass\'{e}, Fan Wei and Liana Yepremyan, in particular, for helpful discussions.

\bibliographystyle{alpha}
\bibliography{snorin}

\end{document}